\documentclass[12pt]{amsart}
\usepackage[pdftex,bookmarks,hypertexnames=true,plainpages=true,colorlinks=true,linkcolor=red,citecolor=blue,urlcolor=green]{hyperref}

\usepackage[top=4cm,bottom=4cm,inner=2.9cm,outer=2.9cm]{geometry}

\newtheorem{thm}{Theorem}[section]
\newtheorem{cor}[thm]{Corollary}
\newtheorem{lem}[thm]{Lemma}
\newtheorem{prop}[thm]{Proposition}

\newtheorem{defn}{Definition}[section]
\newtheorem{rem}{\textbf{Remark}}[section]

\newcommand{\tf}[1]{Theorem~\ref{#1}}

\newcommand{\sef}[1]{Section~\ref{#1}}

\newcommand{\cf}[1]{Corollary~\ref{#1}}

\newtheorem{eg}{\textbf{Example}}[section]
\newtheorem{egs}{\textbf{Examples}}[section]

\newcommand{\nc}{\newcommand}
\nc{\BC}{{\mathbb C}} \nc{\BQ}{{\mathbb Q}} \nc{\BR}{{\mathbb R}}
\nc{\BZ}{{\mathbb Z}} \nc{\BP}{{\mathbb P}} \nc{\BN}{{\mathbb N}}
\nc{\BM}{{\mathbb M}} \nc{\fH}{{\mathcal{H}}} \nc{\vp}{{\varepsilon}}
\nc{\dpa}{{\partial}}\nc{\al}{{\alpha}} \nc{\bt}{{\beta}} \nc{\ga}{{\Gamma}}
\nc{\g}{{\gamma}} \nc{\gt}{{\widetilde{\Gamma}}}\nc{\q}{{\mathcal{Q}}(\ga)}\nc{\si}{\Sigma} \nc{\s}{\sigma}
\nc{\BS}{{\mathbb S}} \nc{\BH}{{\mathbb H}} \nc{\BU}{{\mathbb U}} \nc{\M}{{\mathcal{M}}}

\nc{\Hh}{{\widehat{\mathbb H}}} \nc{\Rr}{{\widehat{\mathbb R}}}
\nc{\mero}{{\mathcal{M}(\BH)}} \nc{\e}{{\mathcal{E}(\ga)}}
\nc{\eh}{{\mathcal{E}_h(\ga)}} \nc{\Cc}{\widehat{\BC}}

\nc{\PZ}{{\mbox{PSL}_2(\BZ)}} \nc{\SZ}{{\mbox{SL}_2(\BZ)}}
\nc{\SR}{{\mbox{SL}_2(\BR)}} \nc{\PR}{{\mbox{PSL}_2(\BR)}}
\nc{\GQ}{{\mbox{GL}_2^+(\BQ)}} \nc{\PQ}{{\mbox{PGL}_2^+(\BQ)}}
\nc{\GR}{{\mbox{GL}_2^+(\BR)}} \nc{\PGR}{{\mbox{PGL}_2^+(\BR)}}
\nc{\GC}{{\mbox{GL}_2(\BC)}} \nc{\La}{\Lambda(\ga)}
\nc{\SC}{{\mbox{SL}_2(\BC)}} \nc{\PC}{{\mbox{PSL}_2(\BC)}}
\nc{\ie}{\textit{ i.e.}} \nc{\ra}{\rightarrow}
\nc{\etc}{\textit{etc}.\@}
\nc{\mt}{M\"{o}bius transformation }
\nc{\mts}{M\"{o}bius transformations}
\nc{\enu}{enumerate}
\begin{document}

\title{Equivariant functions for the M\"{o}bius subgroups  and applications}
\author[]{Hicham Saber}
\address{Department of Mathematics and Statistics, University of Ottawa, Ottawa Ontario K1N 6N5 Canada}
\email{hsabe083@uottawa.ca}

\begin{abstract}
The aim of this paper is to give a generalization of the theory equivariant functions, initiated in \cite{AAS,ElbSeb}, to arbitrary  subgroups of $\PR$.
We show that there is a deep relation between the geometry of these groups and some analytic and algebraic properties of these functions.
As an application, we give a new proof of the classification of automorphic forms for non discrete groups. Also, we prove the following automorphy condition: If $f$ is an automorphic form for a Fuchsian group of the first kind $\ga$, then $f$ has infinitely many non $\ga$-equivalent critical points. \end{abstract}
\maketitle

\section{Introduction}

Let $\BH=\{z\in \BC\,,\  \mbox{Im}(z) > 0\}$ be the upper-half plane of the complex plane, and $\PR=\SR/\{\pm I\}$, where $\SR$ is the group of real $2\times2$ matrices with determinant one. The group $\PR$ acts on $\BH$ by \mts, \ie
$$
\g z= \frac{az+b}{cz+d},~  z\in {\mathbb H}\,,\g= \binom{a\ b}{c\ d}\,\in \SR.
$$

Let $\ga$ be a subgroup of $\PR$, $\mero$ be the set of all meromorphic functions  on $\BH$. Then we have an action of $\ga$ on $\mero$ given by
$$ \g*h=\g^{-1}h\g.
$$
The set of fixed points of this action is
$$\e=\{h\in \mero, \; \g*h=h\ \; \forall \g\in \ga\}=\{h\in \mero, \; \g h=h\g \;\;  \forall \g \in \ga\}.
$$
An element $h$ of $\e$ is called an \textit{an equivariant function} with respect to $\ga$, or a $\ga$-\textit{equivariant function}.

In the papers \cite {s-s} and \cite{ElbSeb}, a serious work was initiated on equivariant functions for the modular group and its subgroups.
 Indeed, they establish the basic properties of these functions, and give many non trivial connections with other topics such as elliptic functions, modular forms, quasi-modular forms, differential forms and sections of line bundles.
Our goal here is  to  generalize  the theory of equivariant functions to arbitrary discrete and non discrete subgroups of $\PR$.
 
 As expected, the geometry of these groups will play a crucial role in our study. For example, an equivariant function for a non elementary group $\ga$ attains any point of the extended plane $\Cc$ infinitely many times once its image contains a point of the limit set of $\ga$. In particular, a meromorphic equivariant function for $\ga$ will have infinitely many non $\ga$-equivalent poles.
Another illustration of this dependence between the geometry of the group and the analytic properties of equivariant functions, will be the complete description of $\e$ when $\ga$ is a non discrete subgroup of $\PR$. When $\ga$ is a Fuchsian group of the first kind, by using the celebrated theorem of Denjoy-Wolff, and inspired by the work of M. Heins, \cite{heins}, we show that the set holomorphic equivariant functions $\eh$ is $\{id_{\BH}\}$.

The importance of the above results lies in the fact that equivariant functions are intimately linked to automorphic forms. Indeed,  to each nonzero automorphic form $f$ on $\ga$ of reel weight $k\neq 0$ and MS $\nu$,  is  attached a meromorphic equivariant function $h_f$  defined by:
\begin{equation}\label{0equiv}
h_f(z)\,=\, z\,+\,k\frac{f(z)}{f'(z)}, ~ z\in {\BH},
\end{equation}
the class of equivariant functions constructed in this way is referred to as  {\it rational equivariant functions}.
It turn out that this class is extremely useful in studying automorphic forms for discrete and non discrete subgroups of $\PR$.
This is being illustrated in \sef{9} where we give a classification of unrestricted automorphic forms for a non discrete subgroup of $\PR$. In fact, the question of classification of these forms has been raised by several authors \cite{knopp1,Ber1, DJ}, and  different methods have been used to provide the answer. Our method relying on equivarience is new and simple.

A second application is to prove the striking result asserting that an  automorphic form of weight $k\neq 0$, and MS $\nu$ on a Fuchsian group of the first kind $\ga$ has infinitely many non $\ga$-equivalent critical points in $\BH$, the case $\ga$ is a modular subgroup was our joined work with A. Sebbar, see \cite{SS}.
This theorem can be thought as a kind of an automorphy test, for example, a direct application of this test shows that the exponential function, or a rational function,  cannot be an automorphic form for a Fuchsian group of the first kind.

The above  automorphy test allowed us again to give  proof of the fact that the Fourier series expansion of an automorphic form cannot have only finitely many nonzero coefficients. This important result is classically proved in a non elementary way using the Rankin-Selberg zeta function of the automorphic form in question, see \cite{KohMas}.

\tableofcontents

\section{\mts\,and subgroups of $\PR$}

In this section we list some known results about \mts\,and subgroups of $\PR$ that will be needed in the sequel. The main references are \cite{Ber} and \cite{katok}.

 Let $\Cc=\BC\cup\{\infty\}$ and $a,b,c,d$ be complex numbers with $ad-bc \neq0$, then the map $g: \Cc\rightarrow \Cc$ defined by
$$
 g(z)=\frac{az+b}{cz+d}
$$
is called a \textit{ M\"{o}bius transformation}. If $\M$ denotes the set of all M\"{o}bius transformations,
then $\M$, equipped with the composition of maps, is a group. Moreover, the map $A\rightarrow g_A$, where
$$
g_A(z)=\frac{az+b}{cz+d}\ ,\ \ A=\binom{a\ b}{c\ d}
$$
is a group isomorphism  $\phi:\PC=\SC/\{-I,I\}\longrightarrow \M$. If $g$ in $\M$, we define
 $$Tr(g)=\mbox{tr}^2(A)$$
 where $A$ is any matrix which projects on $g$. Notice that $Tr(g)$ is invariant under conjugation.

 For each $k\neq0$ in $\BC$, we define $m_k$ by
 $$ m_k(z)=kz~~\mbox{if}~  k\neq1 $$
 and
$$m_1(z)=z+1. $$
Notice that for all $k\neq0$, we have
$$Tr(m_k)=k+\frac{1}{k}+2.$$

\begin{defn}
Let $g\neq I$ be any M\"{o}bius transformation. We say that
\begin{\enu}
\item $g$ is \textit{parabolic} if and only if $g$ is conjugate to $m_1$(equivalently $g$ has a unique fixed point in $\Cc$);
\item $g$ is \textit{loxodromic} if and only if $g$ is conjugate to $m_k$ for some $k$ satisfying $|k|\neq1$ ($g$ has exactly two fixed points in $\Cc$);
\item $g$ is \textit{elliptic} if and only if $g$ is conjugate to $m_k$ for some $k$ satisfying $|k|=1$ ($g$ has exactly two fixed points in $\Cc$).
\end{\enu}
\end{defn}
\begin{defn}
Let $g$ be a loxodromic transformation. We say that $g$ is \textit{hyperbolic} if $g(D)=D$ for some open disc or half-plane $D$ in $\Cc$. Otherwise
$g$ is said to be \textit{strictly loxodromic}.
\end{defn}
\begin{thm}
Let $g\neq I$ be any \mt. Then
\begin{\enu}
\item $g$ is \textit{parabolic} if and only if $Tr(g)=4$;
\item $g$ is \textit{elliptic} if and only if $Tr(g)\in[0,4)$;
\item $g$ is \textit{hyperbolic} if and only if $Tr(g)\in(4,+\infty)$;
\item $g$ is \textit{strictly loxodromic} if and only if $Tr(g)\not\in[0,+\infty)$.
\end{\enu}
\end{thm}
The following result concerns  the iterates of  loxodromic and parabolic \mts.

\begin{thm}
\label{iterates}
 $ $
\begin{\enu}
\item Let $g$ be a parabolic M\"{o}bius transformation  with a fixed  point $\al$. Then for all $z\in \Cc$, $g^n(z)\rightarrow \al$ as $n\rightarrow \infty$; the convergence being
uniform on compact subsets of $\BC-\{\al\}$.
\item Let $g$ be a loxodromic M\"{o}bius transformation. Then the fixed points $\al$ and $\beta$ of $g$ can be labeled so that $g^n(z)\rightarrow \al$ as $n\rightarrow \infty$
for all $z\in \Cc-\{\beta\}$; the convergence being uniform on compact subsets of $\BC-\{\beta\}$.
Here $\al$ is called an \textit{attractive} point and $\bt$ a \textit{repulsive} point.
\end{\enu}
\end{thm}

The group $\PC$ inherit the topology of $\SC$, which is topological space with respect to the standard norm of $\BC^4$. A  subgroup $\ga$ of $\PC$ is called \textit{discrete} if the subspace topology on $\ga$ is the discrete topology. Thus $\ga$ is discrete if and only if for a sequence $T_n$ in $\ga$, $T_n \ra T\in \GC$ implies $T_n=T$ for all sufficiently large $n$.

\begin{defn}
A discrete subgroup of $\PR=\SR/\{-I,I\}$ is called a \textit{Fuchsian} group.
\end{defn}

 $\PR$ is the group of automorphism of $\BH$, and it acts on $\BH$ by \mts:
\begin{equation}
\g z= \frac{az+b}{cz+d},~  z\in {\mathbb H}\,,\g= \pm\binom{a\ b}{c\ d}\,\in \SR.
\end{equation}

\begin{defn}
We say that a subgroup $\ga$ of $\PR$ acts properly discontinuously on $\BH$ if the $\ga-$orbit of any point $z$ in $\BH$ is a discrete subset of $\BH$.
\end{defn}

\begin{thm}
\label{th of dis}
Let $\ga$ be a subgroup of $\PR$. Then $\ga$ is a Fuchsian group if and only if $\ga$ acts properly discontinuously on $\BH$.
\end{thm}

\begin{thm}
\label{th of neib}
Let $\ga$ be a subgroup of $\PR$. Then the following statements are equivalent:
\begin{\enu}
\item $\ga$ acts properly discontinuously on $\BH$;
\item For any compact set $K$ in $\BH$, $T(K)\cap K\neq \emptyset$ for only finitely many $T\in \ga$;
\item Any point $x\in\BH$ has a neighborhood $V$ such that $T(V)\cap V\neq\emptyset$ implies T(x)=x.
\end{\enu}
\end{thm}

\begin{thm}
\label{th 4}
Two non-identity elements of $\PR$ commute if and only if they have the same fixed-point set.
\end{thm}

It is clear that $\PR$ acts on $\Rr=\BR\cup\{\infty\}$, and so we can extend its action on $\BH$ to $\Hh=\BH\cup\BR\cup\{\infty\}$.

\begin{defn}
A subgroup $\ga$ of $\PR$ is called elementary if there exists a finite $\ga$-orbit in $\Hh$.
\end{defn}

\begin{thm}
\label{th6}
 Let $\ga$ be a subgroup of $\PR$ containing  only elliptic elements besides the identity.
 Then all elements of $\ga$ have the same fixed point in $\BH$ and hence $\ga$ is an abelian elementary group.
\end{thm}

\begin{thm} \label{elementary groups}
Any elementary  group is either:
\begin{enumerate}
 \item An abelian group containing beside the identity, only elliptic elements, or only parabolic elements, or only hyperbolic elements.
 \item Conjugate in $\PR$ to a group generated by $h(z)=-1/z$ and by elements of the form $g(z)=lz$ ($l>1$).
 \item Conjugate in $\PR$ to a group whose elements have the form $az+b$, $a,b$ in $\BR$, and contains both parabolic and  hyperbolic elements.
\end{enumerate}

\end{thm}

\begin{cor}
Any elementary Fuchsian group is either cyclic or is conjugate in $\PR$ to a group generated by $g(z)=lz$ ($l>1$) and $h(z)=-1/z$.
\end{cor}

\begin{thm}
\label{th infinite hyperpolic}
A non-elementary subgroup $\ga$ of $\PR$ contains infinitely many hyperbolic elements, no two of which have a common fixed point.
\end{thm}

\begin{thm}
\label{th elementary}
A non-elementary subgroup $\ga$ of $\PR$ is discrete if and only if for each $T$ and $S$ in $\ga$, the group $<T,S>$ is
discrete.
\end{thm}
\begin{thm}

\label{th final}
Let $\ga$ be a non-elementary subgroup of $\PR$. Then the following are equivalent:
\begin{\enu}
\item $\ga$ is discrete;
\item $\ga$ acts properly discontinuously;
\item The fixed points of elliptic elements do not accumulate in $\BH$;
\item The elliptic elements do not accumulate to $I$;
\item Each elliptic element has a finite order.
\end{\enu}
\end{thm}

\begin{defn}
Let $\ga$ be a Fuchsian group. Then the set of all possible limit points of $\ga$-orbits $\ga z$, $z\in\BH$
is called the \textit{limit set} of $\ga$ and denoted by $\Lambda(\ga)$.
\end{defn}

Since $\ga$ is discrete, the $\ga$-orbits do not accumulate in $\BH$. Hence $\Lambda(\ga)\subset \Rr$.

\begin{defn}
$ $
\begin{\enu}
\item If $\La=\Rr$, then $\ga$ is called a Fuchsian group of the first kind.
\item $\La\neq \Rr$, then $\ga$ is called a Fuchsian group of the second kind.
\end{\enu}
\end{defn}

\begin{thm}
\label{th limit set}
If $\La$ contains more than one point, it is the closure of the set of fixed points of the hyperbolic transformations
of $\ga$.
\end{thm}

Another model for the hyperbolic plane, is  the \textit{unit disc}
$$
\BU=\{z\in \BC, |z|<1\}.
$$
The  \mt
\begin{equation}\label{map1}
    \lambda(z)= \frac{z-i}{z+i}
\end{equation}
 is a homeomorphism from $\BH$ onto $\BU$, with the inverse map
\begin{equation}\label{map2}
    \lambda^{-1}(z)= -i\frac{z+1}{z-1}\,,
\end{equation}
$\lambda$ is called the {\textit Cayley transformation}. Note that $\lambda$  maps $\BR$ to  the \textit{principal circle} $\si=\{z\in\BC, |z|=1\}$, which is the Euclidean boundary of $\BU$.
\begin{thm}
\label{elliptic}
The group of automorphisms of $\BU$ is  $\lambda\PR \lambda^{-1}$, and its elements have the following form:
$$
     \frac{az+\bar{c}}{cz+\bar{a}},\mbox{ where }\; a,c\in\BC,\;a\bar{a}-c\bar{c}=1.
$$
\end{thm}

\section{Equivariant functions}
\label{3}
Let $\ga$ be a subgroup of $\PR$, $\mero$ be the set of all meromorphic functions  on $\BH$. Then we have an action of $\ga$ on $\mero$ given by
$$ \g*h=\g^{-1}h\g.
$$
The set of fixed points of this action is
$$\e=\{h\in \mero, \; \g*h=h\ \; \forall \g\in \ga\}=\{h\in \mero, \; \g h=h\g \;\;  \forall \g \in \ga\}.
$$

\begin{defn}$ $
An element $h$ of $\e$ is called \textit{an equivariant function} with respect to $\ga$, or a $\ga$-\textit{equivariant function}.
\end{defn}

 \begin{rem}
 $\eh$ will be the set of  holomorphic equivariant function.
\end{rem}

\begin{egs}
\
\begin{\enu}
\item The identity map is a trivial example of an equivariant functions, it will be denoted by $h_0$.
\item If $\ga$ is an abelian group, then $\ga$ is contained in $\e$.
\end{\enu}
Infinitely many non trivial examples will be constructed in  section \ref{construction of equi functions}.
\end{egs}

It turns out that these equivariant functions have very interesting mapping properties. One major question that arises is about the
size of the image set of an equivariant function. Our purpose is to give some necessary conditions for an equivariant function to be surjective, and since commuting with the action of the group is not a trivial property, it is natural that these conditions will involve some characteristics of the group, for example, the nature of the limit set $\La$ of $\ga$.

\begin{defn} $ $
\begin{\enu}
\item Let $\g$ be an element of $\SC$, we define $Fix(\g)$ to be the set  fixed points of $\g$
\item Let $\ga$ be a subgroup of $\PC$, we define $Fix(\ga,\Cc)$ to be the set of  fixed points of $\ga$. Notice that
 $Fix(\ga)$ is contained in $\Cc$.
\item Let $\ga$ be a subgroup of $\PR$, we define $Fix(\ga)$ to be the set of parabolic or hyperbolic fixed points of $\ga$. Notice that
 $Fix(\ga)$ is contained in $\Rr$.
\item Two points $z,w\in\Cc$ are said to be \textit{$\ga$-equivalent} if there exists $\g \in\ga$ such that $w=\g(z)$.
\end{\enu}
\end{defn}

\begin{thm}
Let $\ga$ be a discrete subgroup of $\PR$, $h\in \e$, $x \in Fix(\ga)\cap Im(h)$, $\g\in \ga$ with $\g x=x$. If $z\in \Cc$ is not a fixed point $\g$, then
$h^{-1}(\{z\})$ contains infinitely many non $\ga-$equivalent points.
\end{thm}

\begin{proof}
\
Suppose that $x=h(w)$ for some $w\in \BH$, since $x \in Fix(\ga)\subset\Rr$, $w$ is not an elliptic point. Indeed, if $\g w=w$ for some elliptic element $\g$ in $\ga$, then
$$
x=h(w)=h(\g w)=\g h(w)=\g x,
$$
and the elliptic element $\g$ will have a fixed point in $\Rr$, which is impossible.
Using \tf{th of neib} and the fact that $ w$ is not elliptic, we get an open neighborhood $V$ of $x$ containing no $\ga$-equivalent points. After intersecting $V$ with an open neighborhood obtained from the local openness of $h$ at $w$, we may assume that $h(V)$ is an open neighborhood of $x$.

 Suppose that $\g$ is parabolic and  $z\in \Cc-\{x\}$. Using \tf{iterates}, there exists $N>0$ such that $\g^n z \in h(V)$ for all $n>N$.
 Write $\g^nz=h(w_n)$ for some $w_n\in V$, then
$$
z=h(\g^{-n}w_n) \mbox{ for all } n>N.
$$

Suppose that for some $m,n>N$ we have
$$
\g^{-n}w_n=\al\g^{-m}w_m \mbox{ with } \al \in\ga,
$$
then $\g^{m-n}=\al$ since $V$ contains no $\ga$-equivalent points, and so $\al x=x$.
Also, we have $h(\g^{-n}w_n)=h(\al\g^{-m}w_m)$, which implies that $\g^{-n}h(w_n)=\al\g^{-m}h(w_m)$. This gives $z=\al z$ and hence $z=x$, which is impossible since $\{x\}$ is the set of fixed point of $\g$.

Suppose that $\g$ is hyperbolic. If $x$ is a repulsive point of $\g$, then it is an attractive one for $\g^{-1}$, which is also a hyperbolic element of $\ga$ verifying all the hypotheses of the theorem, and so we can always assume that $x$ is an attractive point for $\g$. Now, use \tf{iterates} and  mimic the proof of the first part to get the desired result.
\end{proof}

If in the above theorem $\ga$ was not an elementary group, then  we have:

 \begin{cor}
Let $\ga$ be  a non elementary discrete subgroup of $\PR$, $h\in \e$, $x \in Fix(\ga)\cap Im(h)$, $\g\in \ga$ with $\g x=x$. Then $h: \BH\ra \Cc$ is surjective. More precisely, for any $z\in \Cc$, $h^{-1}(\{z\})$ contains infinitely many non $\ga-$equivalent points.
\end{cor}
\begin{proof}
\
 We only need to show that $h^{-1}(\{z\})$ contains infinitely many non $\ga-$equivalent points when $z$ is a fixed point of $\gamma$.  By
 \tf{th infinite hyperpolic}, $\ga$  contains infinitely many hyperbolic elements no two of which have a common fixed point.  Take any hyperbolic element $\g_1\in \ga-\{\g\}$, since $\g$ and $\g_1$ have no  common fixed point, the pre-image of any fixed point of $\g_1$ will contain infinitely many non $\ga$-equivalent points. In particular $\g_1$ verifies the conditions of the above theorem, and so $h^{-1}(\{z\})$ contains infinitely many non $\ga-$equivalent points when $z$ is a fixed point of $\gamma$( since $z$ is  not  a fixed point of $\g_1$).
\end{proof}

  When $\ga$ is a non elementary Fuchsian group, the closure of the set of fixed points of the hyperbolic transformations of $\ga$ will be  $\La$. Therefore, if the image of $h$ contains one point in $\La$, then it will contain infinitely many hyperbolic fixed points of $\ga$. As a consequence, we have

\begin{cor}
\label{cor1}
Let $\ga$ be a non elementary Fuchsian subgroup of $\PR$, $h\in \e$. If $ \La \cap Im(h)$ is non empty, then $h: \BH\ra \Cc$ is surjective, and for any $z\in \Cc$ the set $h^{-1}(\{z\})$ contains infinitely many non $\ga-$equivalent points.
\end{cor}

\section{Meromorphic equivariant functions of non discrete subgroups of $\PR$}
Our aim here is to give a complete description of $\e$ and $\eh$.

We begin by the following useful lemma, which can be directly deduced from the definition of an equivariant function.
\begin{lem}
Let $\ga$ be a subgroup of $\PR$, $h\in \e$, and $\al \in \PR$. Then
$$\al^{-1}h\al \in \mathcal{E}(\al^{-1}\ga \al).
$$
\end{lem}

In our investigation of $\e$, we start first with  the  case when  $\ga$ is a non discrete non elementary subgroup of $\PR$.

\begin{thm}\label{th classification 1}
Let $\ga$ be a non discrete non elementary subgroup of $\PR$. Then
$$\e=\{h_0\}.$$
\end{thm}

\begin{proof}
  Let $h\in \e$. By \tf{th elementary} and \tf{th final}, $\ga$ contains an elliptic element $\g$ of infinite order.
By the above lemma and  the transitivity of  the action of $\PR$ on $\BH$,  we can suppose that $\g$ fixes $i$. If $\lambda$ is the Cayley transformation, see Equation \ref{map1}, then the map
$$ g=\lambda h \lambda^{-1}:\BU \ra \Cc
$$
commutes with the action of the group $\ga_1= \lambda\ga \lambda^{-1}$ on the unit disc $\BU$ . Moreover, the element $\al=\lambda \g \lambda^{-1}\in\ga_1$ is elliptic of infinite order and fixing $0$. Hence, by \tf{elliptic}, we have
$$\al(z)=kz\,,~~k=e^{2\pi i\theta}~  (\theta \mbox{ is irrational}).
$$
As in the proof of \tf{th final}, we can construct a sequence $\{\al^{a_n}, \; a_n\in\BN\}$ such that  $\al^{a_n}\ra I$ as $n\ra\infty$.

Since $\al \in \ga_1$, we have $g\al=\al g$,\ie
$$
g(kz)=kg(z)\mbox{ for all } z\in\BU.
$$
Differentiating both sides yields
$$
g'(kz)=g'(z) \mbox{ for all } z\in\BU.
$$

Now, let  $x\in \BU-\{0\}$ be any point which is not a pole of $g$, $G$ be  the meromorphic function $g^{'}-g'(x)$, and $x_n=\al^{a_n}x=k^{a_n}x$. We have
$$
G(x)=0,~G(x_n)=0, \mbox{ and } x_n\ra x \mbox{ as } n\ra\infty.
$$
Therefore, $G$ is meromorphic on $\BU$ with its set of zeros having an accumulation point in $\BU$. It follows that
 $G$ must be a constant function, that is, for some $a,b\in\BC$
$$
g(z)=az+b\,, \ z\in\BU \, ,
$$
then  commutativity of $g$ with $\al$ implies that $b=0$.

Since $\ga$ is not elementary, $\ga_1$ contains an element $\bt$ which fixes neither $0$ nor $\infty$. If $y$ is a fixed point of $\bt$,
then the commutativity of $\bt$ and $g$ implies that  $g(y)=ay$ is also a fixed point of $\bt$. Iterating this, one shows that $a^2y$ is also a fixed point of $\bt$.
Therefore, two of the complex numbers $y$, $ay$, $a^2y$ must be equal. Since $y$ is neither 0 nor $\infty$, one easily sees that $a=\pm 1$.
The case $a=-1$ is  excluded, otherwise we will have
$$
\g_1(-z)=-\g_1(z)\,,\  z\in\BU, \mbox{for all}\, \g_1\in\ga_1.
$$
 This implies that $\g_1(0)=0$ for all  $\g_1\in\ga_1$,
which impossible since $\ga_1$ is not elementary. Hence $g$ is equal to the identity map, and so that $h=h_0$.
\end{proof}

 Now suppose that $\ga$ is a non discrete elementary subgroup of $\PR$. Then by \tf{elementary groups}, the set $Fix(\ga,\Cc)$ of fixed points of $\ga$  contains at most two points when $\ga$ is abelian.

\begin{thm}\label{th classification 2}
Let $\ga$ be a non discrete elementary subgroup of $\PR$. Then
\begin{enumerate}
\item If $\ga$ is abelian, then $$\e=Fix(\ga,\Cc)\, \cup \, \big\{ \g \in \PC;\, Fix(\g)=Fix(\ga,\Cc) \, \big\}.$$
\item If $\ga$ is not abelian, then $$\e= \big\{ h_0 \big\},$$  or there exists $\g$ in $\PR$ such that $$\e= \big\{ h_0,\, \g^{-1}(-h_0)\g \big\}.$$
\end{enumerate}

\end{thm}

\begin{proof}
Suppose that $\ga$ is abelian. If $\ga$  contains  beside the identity only elliptic elements, then $Fix(\ga,\Cc)=\{\al,\, \overline{\al} \}$ for some  $\al \in \BH$. Thus the constant functions $h(z)=\al$ and $h(z)=\overline{\al}$ are in $\e$.
As usual,  suppose that we are working in the unit disc model and that $\al=0$. Then any $\g \in \ga$  has the form $\g(z)=kz$ for some $ k \in \BC$.

Let  $h$ be an element of $ \e$. Since $\ga$ is not discrete, we can construct a sequence $\{\g_n\}$ of elements of $\ga$ such that  $\g_n\ra I$ as $n\ra\infty$, and using the same method
as in the proof of the above theorem, we find that for some $ a \in \BC$,
$$
h(z)=az~,~~   z\in \BU,
$$
and so,  in the upper half plane model, $h$ is an element of
$$\big\{ \g \in \PC;\, Fix(\g)=Fix(\ga,\Cc) \, \big\}.$$
Hence
$$
\e \subseteq Fix(\ga,\Cc)\cup \big\{ \g \in \PC;\, Fix(\g)=Fix(\ga,\Cc) \, \big\}.
$$
The other inclusion is clear.

If $\ga$  contains  beside the identity only parabolic elements, then $Fix(\ga,\Cc)=\{\al \}$ for some  $\al \in \Rr$. Thus the constant function $h(z)=\al$ is in $\e$. After a conjugation in $\PR$  we can take $\al=\infty$, and so any $\g \in \ga$ will have the following form
$$
\g(z)=z+b~,~~ z\in \BH
$$
 for some $ b \in \BR$.  As in the proof of the preceding theorem, we get
$$
h(z)=cz+d~,~~   z\in \BH
$$
for some $ c,d \in \BC$.
The commutativity of $h$ with the elements of $\ga$ implies that $c=1$. Hence $h$ has the desired form, and we conclude as in the preceding case.

If $\ga$  contains  beside the identity only hyperbolic elements, then $Fix(\ga,\Cc)=\{\al, \bt \}$ for some  $\al, \bt \in \Rr$. Thus the constant functions $h(z)=\al$ and $h(z)=\bt$ are in $\e$. After a conjugation in $\PR$, we can take $\al=\infty$ and $\bt=0$, and so any $\g \in \ga$ will have the following form
$$
\g(z)=az~,~~ z\in \BH
$$
 for some $ a \in \BR$.  As in the proof of the preceding theorem, we get
$$
h(z)=cz+d~,~~   z\in \BH
$$
for some $ c,d \in \BC$.
The commutativity of $h$ with the elements of $\ga$ implies that $d=0$, and we conclude as in the preceding case.

Now, Suppose that $\ga$ is not abelian. Then by \tf{elementary groups}, we have the following  two cases:

$\bullet$ $\ga$  is conjugated  in $\PR$ to a group whose elements have the form $az+b$, $a,b$ in $\BR$, and contains both parabolic and  hyperbolic elements. Similarly, we find that
$$
h(z)=cz+d~,~~   z\in \BH
$$
for some $ c,d \in \BC$.
The commutativity of $h$ with the parabolic and  hyperbolic elements of $\ga$ implies that $c=1$ and $d=0$, that is $h=h_0$, hence the the first part of the result.

$\bullet$  $\ga$  is conjugated in $\PR$ to a group generated by $h(z)=-1/z$ and by elements of the form $g(z)=lz$ ($l>1$).
This implies that an element $\g \in\ga$ will be either of the form
$$
\g(z)=az~ \mbox{ for all } z\in \BH, \mbox{ for some } a \in \BR,
$$
and so it is a hyperbolic.
Or of the  form
$$
\g(z)=b/z~\mbox{ for all } z\in \BH, \mbox{ for some } b \in \BR-\{0\},
$$
hence it is an elliptic element of order two. Since $\ga$ is not discrete, we have a sequence $\{\g_n\}$ of elements of $\ga$ such that $\g_n\ra I$ as $n\ra\infty$. If $\{\g_n\}$ contains infinitely many elliptic elements, then it has a subsequence $\{\g_{n_m}\}$ with
$$
\g_n(z)=b_{n_m}/z~\mbox{ for all } z\in \BH, \mbox{ for some } b_{n_m} \in \BR-\{0\},
$$
such that $\g_{n_m}\ra I$ as $m\ra\infty$. Then, for all $z\in \BH$ we have $\g_{n_m}z\ra z$ as $m\ra\infty$, \ie\, $b_{n_m}\ra z^2$ as $m\ra\infty$ for all $z\in \BH$, which is absurd. Hence $\{\g_n\}$ contains only finitely many elliptic elements, thus we may assume that all elements in $\{\g_n\}$ are hyperbolic.
Finally, as in the proof of the preceding theorem, we find that for some $c \in \BC$
$$
h(z)=cz~ \mbox{ for all } z\in \BH.
$$
and the commutativity  with the elliptic elements shows that $c=\pm 1$. Hence $h=\pm h_0$ as desired.

\end{proof}

\section{Holomorphic equivariant functions of Fuchsian group of the first kind}

Unlike the non discrete case, the classification of holomorphic equivariant functions when $\ga$ is a Fuchsian groups presents
more challenges and the proof is more elaborate. In addition to the equivarience and the holomorphy of the functions, we have to use
some properties of the equivarience group, namely the limit set.

 The fact that the limit set of a Fuchsian group of the first kind $\ga$ is $\Rr $ enables us to show that $\eh=\{h_0\}$. But in the second kind case, or in the elementary  the situation is more difficult, since the limit set of Fuchsian groups of the second kind is complicated, and it is  trivial for an elementary Fuchsian groups (consisting at most of two points).

\begin{thm} \label{th classification 3}
If $\ga$ is a Fuchsian group of the first kind, then $\eh=\{h_0\}$.
\end{thm}

\begin{proof}

Suppose that $\ga$ is a Fuchsian group of the first kind, and let $h$ be in $\eh$.
Since $h$ is an open map, and $\La=\Rr$, $h$ cannot take any real value, otherwise it will have infinitely many poles, see \cf{cor1}. Hence either $h$ maps the upper half plane into itself, or it maps it into the lower half plane $\BH^-$.

Suppose first that $h$ maps the upper half-plane $\BH$ into itself and that $h\neq h_0$. According to the theorem of Denjoy-Wolff, \cite{carl}, $h$ is either an elliptic element of $\PR$ or the iterates $h_n = h\circ...\circ h$, $n$ times, which are also in $\eh$,
converge  uniformly on compact subsets of $\BH$ to a point $p \in \Hh$. Since $h$ commutes with $\ga$, $h$ cannot be an elliptic element of $\PR$, otherwise all elements of $\ga$ will have the same fixed point set according to \tf{th 4} and therefore $\ga$ will be an elementary group
which is not the case.  If the iterates $h_n$ of $h$ converge to a point $p\in \Hh$, then since
$\g h_n=h_n\g$, we have $\g p=p$ for all $\g\in \ga$. This impossible since $\ga$ is not an elementary group. Consequently, the only equivariant
 function that maps $\BH$ into itself is $h_0$.

 We now suppose that $h$  maps the upper half-plane $\BH$ into the lower half-plane  $\BH^-$. But then we can extend $h$ to a holomorphic function $\widetilde{h}$ defined on $\BH \cup \BH^-$ by
$$\widetilde{h}(z)=\overline{h(\bar{z})}, ~z\in \BH^-,
$$
and it is easy to see that the restriction of $\widetilde{h}\circ \widetilde{h}$ to $\BH$ is an equivariant function that maps $\BH$ into itself. By the above discussion, we have
$$\widetilde{h}\circ \widetilde{h}(z)=z~\mbox{ for } z\in\BH\,.$$
This implies that $h$ is bijective. Therefore, the function $-h$ is an automorphism of $\BH$ and so is an element of $\PR$, and so  $\widetilde h$ extends to  a \mt commuting with $\ga$. Since $\widetilde h\circ \widetilde h=id$, $\widetilde h$ is elliptic of order $2$. After a suitable conjugation, we can assume that $\widetilde h=m_k$ with $k^2=1$.
 If $k=-1$ then for all $\g \in\ga$ we have
$$
\g(-z)=-\g(z)\mbox{ for all } z\in\Cc.
$$
In particular,
$$
\g(0)=0 \mbox{ for all } \g\in\ga,
$$
and hence $\ga$ is elementary which is not the case. Thus $k=1$, and so $h=h_0$. We conclude that $\eh=\{h_0\}$.
\end{proof}

Note that the unique property of $\ga$ used above is $\ga$ being non elementary, hence we have

\begin{cor}
If $\ga$ is a non elementary Fuchsian group and $h\in\eh-\{h_0\}$, then the image of $h$ meets the real line.
\end{cor}

%--------------------------------------------------------------------------------------------------------------------------------------------------------
%--------------------------------------------------------------------------------------------------------------------------------------------------------
%--------------------------------------------------------------------------------------------------------------------------------------------------------
%--------------------------------------------------------------------------------------------------------------------------------------------------------
%--------------------------------------------------------------------------------------------------------------------------------------------------------
%--------------------------------------------------------------------------------------------------------------------------------------------------------
%--------------------------------------------------------------------------------------------------------------------------------------------------------
%--------------------------------------------------------------------------------------------------------------------------------------------------------
%--------------------------------------------------------------------------------------------------------------------------------------------------------

\section{Unrestricted automorphic forms}
In this section we follow the treatment of \cite{RN}, and \cite{knopp}.

Let $\g=\left(
         \begin{array}{cc}
           a & b \\
           c & d \\
         \end{array}
       \right) \in \SR$, $z\in\BH$ and $k$ be  a real number. If $j_\g(z)=j(\g,z)=cz+d$, then for any complex valued function $f$ defined on $\BH$, the slash operator of weight $k$ on $f$  is defined by:
$$
(f|_k\g)(z)=j_\g(z)^{-k}f(\g z).
$$
\begin{defn}
An \textit{automorphic factor} (AF) $\mu$ of weight $k\in \BR$ for a subgroup $\ga$ of $\PR$ is a map $\mu: \ga \times \BH \ra \BC^{\times}$ satisfying
\begin{\enu}
\item For all $\g=\left(
         \begin{array}{cc}
           a & b \\
           c & d \\
         \end{array}
       \right) \in \ga$ and $z \in \BH$, $|\mu_{\g}(z)| = |j_\g(z)|^k$.
\item For all $\al,\g \in \ga$ and $z \in \BH$, $\mu_{\al\g}(z) =\mu_{\al}(\g z)\mu_\g(z).$
\item For all $\g \in \ga$ and $z \in \BH$,  $\mu_{-\g}(z) =\mu_{\g}(z)$
\end{\enu}
\end{defn}

\begin{rem}
In what follows we will identify an element $\g \in \PR$ with its representative $\left(
         \begin{array}{cc}
           a & b \\
           c & d \\
         \end{array}
       \right) \in \SR$, since the third condition says that the automorphic factor $\mu$ is well defined on $\ga \times \BH$, where $\ga$ is a subgroup of $\PR$.
\end{rem}

\begin{eg}
For any $k\in 2\BZ$, the function $j^k_\g: \PR \times \BH \ra \BC^{\times}$ is an an automorphic factor of weight $k$.
\end{eg}

 Since a holomorphic function on $\BH$ of constant modulus must be constant, then from the above definition we have
$$
\mu_\g(z)=\mu(\g,z)=\nu(\g)j_\g
$$
for all $\g \in \ga$ and $z \in \BH$, where $\nu(\g)$ depends only on $\g$ and
$$
|\nu(\g)|=1.
$$
The factor $\nu(\g)$ is called a \textit{multiplier}, and the function $\nu$ defined on $\ga$ is called a \textit{multiplier system}(MS) of weight $k$ for $\ga$. Note that $\nu(I) = 1$ and $\nu(-I) = e^{-i\pi k}$, and for any $\al, \g \in \ga$ we have
$$
\nu(\al\g)=\s(\al,\g)\nu(\al)\nu(\g), \mbox{ where } \s(\al,\g)=\frac{j(\al,\g z)j(\g,z)}{j(\al\g,z)}~,~~(|\s(\al,\g)|=1).
$$

\begin{defn}
Let $\mu$ be an automorphic factor of weight $k \in \BR$ for a subgroup $\ga$ of $\PR$ and let $\nu$ be the associated multiplier system. A function $f : \BH \ra \BC$ is called an \textit{unrestricted automorphic form} for $\ga$ of weight $k$, with automorphic factor $\mu$ (or, equivalently, with
multiplier system $\nu$) if it satisfies
\begin{\enu}
\item $f$ is meromorphic on $\BH$.
\item $f|_k\g =\nu(\g)f$ for all $\g \in \ga$.
\end{\enu}
 The $\BC-$vector space of all unrestricted automorphic forms for $\ga$ of weight $k$ and MS $\nu$ will be denoted by $M_u(\ga,k,\nu)$.\\
\end{defn}
 
Suppose that $\g$ is a parabolic element of $\ga$ with  fixed point $x$, and let
 $$L=L(x)=\left(
              \begin{array}{cc}
               0 &-1 \\
               1 & -x \\
              \end{array}
              \right).$$
Then $L\in\SR$, $Lx=\infty$, and $U_L=L^{-1}\g L$ is an element of $\ga^L=L^{-1}\ga L$ fixing $\infty$, and so
$$ L^{-1}\g L(z)=z+n_L, \mbox{ for some } n_L \in \BR.
$$
Since $U_L$ is a translation we have
$$
\nu^L(U_L)=\nu(\g)=e^{2\pi i k_L} \mbox{ with } 0\leq k_L < 1.
$$
 Using the fact that $f|_kL$ is an unrestricted automorphic From  on $\ga_L$ we get
$$
(f|_kL)(z+n_L)= e^{2\pi i k_L}(f|_kL)(z), ~z\in \BH.
$$
It follows that the  function
$$
f_L^*(z)=e^{-2\pi i k_L z/n_L}(f|_kL)(z),~z\in \BH,
$$
is periodic of period $n_L$. Hence, there is a unique meromorphic function $F_L$ defined on $\BU-\{0\}$, such that
$$
f_L^*(z)=F_L(q),
$$
 where $ q=e^{-2\pi i z/n_L}$ is  called a \textit{local uniformizing variable}. If $f_L$ is holomorphic on $\{z\in \BH,~ \mbox{Im}(z)>\eta\}$, for some $\eta\geq0$, then $F_L$ is holomorphic on $\{t\in \BU,~ 0<|q|<e^{-2\pi i \eta/n_L}\}$ and has a convergent Laurent series in this neighborhood of $0$
$$
(f|_kL)(z)= e^{2\pi ik_L z/n_L} \sum_{m=N_L}^{\infty} a_m(L)e^{2\pi i mz/n_L},~ N_L\in\BZ\cup\{\infty\}.
$$
 We say that $f$ is meromorphic, holomorphic at the cusp $x$, if respectively $N_L\neq\infty$, $N_L \geq 0$.
\begin{defn}
Suppose that $f\in M_u(\ga,k,\nu)$. Then
\begin{enumerate}
 \item If $f$  is  meromorphic  at every cusp of $\ga$, then  $f$ is called an \textit{automorphic form} of weight $k$ and MS $\nu$ for $\ga$ . The $\BC$ vector space of these functions is denoted by $M(\ga,k,\nu)$.

\item If $f$ is holomorphic on $\BH$ and at every cusp of $\ga$, then $f$ is called an \textit{entire automorphic form} of weight $k$ and MS $\nu$ for
      $\ga$. The $\BC-$vector space of entire automorphic forms is denoted by $M_e(\ga,k,\nu)$.

\item If $f$ is an entire automorphic form vanishing at all the cusps of $\ga$, then $f$ is called a \textit{cusp form} of weight $k$
       and MS $\nu$ for$\ga$. The $\BC-$vector space of entire automorphic forms is denoted by  $S(\ga,k,\nu)$.
\end{enumerate}

In all cases, when $k=0$ and $\nu=1$, the word form is replaced by the word function.

\end{defn}
If the context is clear, the reference to the group and the multiplier system will be omitted.

Examples of automorphic forms of a Fuchsian group of any type can be found in \cite{FRD}. Here, we content ourselves to give examples in the case of the modular group and some of its special subgroups.

Since the limit set of a Fuchsian group is closed and contains the parabolic points, the limit set of  $\PZ$ is equal to $\BR \cup \{\infty\}$ as
 the set of cusps  of $\PZ$ is $\BQ \cup \{\infty\}$. Hence $\PZ$ and all its subgroups of finite index are Fuchsian groups of the first kind.

The modular group $\PZ$ is generated by $S$ and $T$ or  by $S$ and $P$, where
$$T=\left(
    \begin{array}{cc}
      1 & 1 \\
      0 & 1 \\
    \end{array}
  \right),\;S=
\left(
  \begin{array}{cc}
    0 & -1 \\
     1& 0 \\
  \end{array}
\right),\; P=ST=
\left(
  \begin{array}{cc}
    0 & -1 \\
    1 & 1 \\
  \end{array}
\right).
$$
Moreover, $S$ is elliptic of order $2$ with fixed point $i$, and $P$ is elliptic of order $3$ with fixed point $\rho=e^{2\pi i/3}$.

An important class of subgroups of $\PZ$ is the so-called \textit{congruence subgroups}.  These are  subgroups $\ga$ of $\PZ$  containing a principal congruence subgroup $\ga(n)$, $n$ is a positive integer, with
$$
\ga(n)=\{\g \in \SZ,\; \g\equiv \pm I \mod n\}/\{\pm I\},
$$
and the smallest such $n$ is called the \textit{level} of $\ga$.

Some well known examples of congruence subgroups of $\PZ$ are :
$$
\ga_1(n)=\{ \g \in \SZ,\;\g\equiv \pm
 \left(
             \begin{array}{cc}
               1 & * \\
               0 & 1 \\
             \end{array}
           \right) \mbox{ mod } n \}/\{\pm I\},
$$

$$
\ga_0(n)= \left\{\left(
             \begin{array}{cc}
               a & b \\
               c & d \\
             \end{array}
           \right)
 \in \SZ,\; c\equiv 0 \mbox{ mod } n \right\}/\{\pm I\},
$$

$$
\ga^0(n)=\left\{\left(
             \begin{array}{cc}
               a & b \\
               c & d \\
             \end{array}
           \right)
 \in \SZ,\; b\equiv 0 \mbox{ mod } n\right\}/\{\pm I\}.
$$
It is clear that
$ \ga(n)\subset\ga_1(n)\subset \ga_0(n)$
and that  $\ga_0(n)$, $\ga_1(n)$, $\ga^0(n)$ are of level $n$. Note that $\ga_0(n)$ is conjugate to $\ga^0(n)$ by $\binom{n\ 0}{0\ 1}$.

When speaking about the modular group $\PZ$ and its subgroups, the word automorphic in the definition  is replaced by the word modular.
We start our examples by the well known   Eisenstein series. They are defined for every even integer $k\geq 2$ and $z \in \BH$ by
$$G_k(z) =\sum_{m,n}\,'\frac{1}{(mz + n)^k},
$$
where the symbol $\sum\;'$ means that the summation is over the pairs $(m, n) \neq (0, 0)$. Note that this series is not absolutely convergent for k = 2.
We normalize Eisenstein series by letting
$$
E_k=2\zeta(k)G_k
$$
to get the following representation
$$
E_k( z ) \,=\, 1-\frac{2k}{B_k} \,\sum_{n=1}^{\infty}\,\sigma_{k-1}(n)q^n,\;  q=e^{2\pi i z}.
$$
Here $B_k$ is the $k$-th \textit{Bernoulli} number and $\sigma_{k-1}(n)=\sum_{d|n}d^k$.
The most familiar Eisenstein series are:
$$
E_2( z ) \,=\, 1-24\,\sum_{n=1}^{\infty}\,\sigma_1(n)q^n,
$$
$$
E_4(z)\,=\,1+240\,\sum_{n=1}^{\infty}\,\sigma_3(n)q^n,
$$
$$
E_6(z)\,=\,1-504\,\sum_{n=1}^{\infty}\,\sigma_5(n)q^n.
$$
For $k\geq 4$, the series $E_k$ are entire modular forms of weight $k$. The Eisenstein series $E_2$ is holomorphic on $\BH$ and at the cusps, but it is not a modular form as it does not satisfy the modularity condition. The Eisenstein series $E_2$ is an example of a \textit{quasimodular form}  and plays an important role in the construction of equivariant functions as will be seen later on. Moreover, $E_2$ satisfies
\begin{equation}\label{log-div}
E_2(z)\,=\,\frac{1}{2\pi i}\frac{\Delta'(z)}{\Delta(z)},
\end{equation}
where $\Delta$ is the weight 12 cusp form for $\PZ$ given by
$$
\Delta(z)\,=\,q\,\prod_{n\geq 1}\,(1-q^n)^{24}.
$$
The Eisenstein series satisfy  the Ramanujan relations
\begin{equation}\label{ram1}
\frac{6}{\pi i}\,E'_2\,=\,E_2^2-E_4\, ,\
\end{equation}
\begin{equation}\label{ram2}
 \frac{3}{2\pi i}\,E'_4\,=\,E_4E_2-E_6\, ,
 \end{equation}
\begin{equation}\label{ram3}
  \frac{1}{\pi i}\,E'_6\,=\,E_6E_2-E_4^2\,.
\end{equation}

The Dedekind j-function given by
$$
j(z)=\frac{E_4^3-E_6^2}{\Delta}
$$
is a modular function and it generates the function field of modular functions for $\PZ$.

An example of an entire modular form with a non-trivial multiplier system is the Dedekind eta-function $\eta(z)=\Delta(z)^{1/24}$. It is a weight $1/2$ entire modular form for $\PZ$ with the multiplier systems  given by
$$
\nu(T)=e^{i\pi/12}, \;\nu(S)=e^{-i\pi/4},\; \nu(P)=e^{-i\pi/8}.
$$

Now, we give an example of entire modular forms on congruence subgroups with non-trivial multiplier systems, namely
the Jacobi theta functions. They are defined by
$$\begin{array}{rcl}
\theta_2(z)&=&\displaystyle\sum_{n=-\infty}^{\infty}q^{(n+1/2)^2}\, ,\\
&&\\
\theta_3(z)&=&\displaystyle\sum_{n=-\infty}^{\infty}q^{n^2}\, ,\\
&&\\
\theta_4(z)&=&\displaystyle\sum_{n=-\infty}^{\infty}(-1)^nq^{n^2}\,.
\end{array}
$$
They are entire holomorphic modular
forms of weight $\frac{1}{2}$ for the conjugate congruence subgroups $\ga_0(2),\; P^{-1}\ga_0(2)P,\; P^{-2}\ga_0(2)P^2$ respectively. Their associated multiplier systems are $u, v$ and $w$ respectively and are defined by
$$
 v(-I)= u(-I)=  w(-I)=  -i,
$$
$$
v(T^2)=  u(P T^2 P^{-1})=  w(P^2 T^2 P^{-2})=  1,
$$
$$
 v(S)=  u(P S P^{-1})=  w(P^2 S P^{-2})=  e^{-i\pi/4}.
$$

%************verifie le dernier**************

Moreover, these modular forms do not vanish on $\BH$ and satisfy the Jacobi identity
$$
\theta_2^4+\theta_4^4=\theta_3^4.
$$
Furthermore, the following relations hold between the theta functions, $E_4$, $\Delta$ and  $\eta$:
\begin{equation}\label{eta}
\begin{array}{rcc}
\Delta(z)&=&\displaystyle(2^{-1}\theta_2(z) \theta_3(z) \theta_4(z))^8\\
&&\\
E_4(z)&=&\displaystyle\frac{1}{2}(\theta_2^8(z)+\theta_3^8(z)+ \theta_4^8(z))\\
&&\\
\theta_2(z)&=&2\displaystyle\frac{\eta(4z)^2}{\eta(2z)}\\
&&\\
\theta_3(z)&=&\displaystyle\frac{\eta(2z)^5}{\eta(z)^2\eta(4z)^2}\\
&&\\
\theta_4(z)&=&\displaystyle\frac{\eta(z)^2}{\eta(2z)}.
\end{array}
 \end{equation}

\section{unrestricted  quasi automorphic forms}

Let $\ga$ be a subgroup of $\PR$, $k\in \BR$, and $\nu$ be a MS for $\ga$. An \textit{unrestricted quasi-automorphic form} of weight $k$, depth $p$  and MS $\nu$ on $\ga$ is a meromorphic function $f$ on $\BH$ such that for all $\g\in\ga$, $z\in\BH$
\begin{equation}\label{quasi}
(f|_k\g)(z)\,=\nu(\g)\,\sum_{i=0}^{p}\,f_n(z)\left(\frac{c}{cz+d}\right)^n,
\end{equation}
where each $f_n$ is a meromorphic function of $\mathbb H$.
 The quasi-automorphic polynomial attached to $f$ is defined by
\begin{equation}\label{quas-pol}
P_{f,z}(X)\,:=\,\sum_{i=0}^{p}\,f_n(z)X^n.
\end{equation}
The vector space of such functions will be  denoted by $\widetilde{M_u}(\ga,k,p,\nu)$, and  $\widetilde{M_u}(\ga,k,\nu)=\bigcup_p \widetilde{M_u}(\ga,k,p,\nu)$ will denote the space of all unrestricted quasi-automorphic forms of weight $k$ and MS $\nu$. When $k$ runs over $\BZ$, the set $\widetilde{M_u}(\ga,\nu)=\bigoplus_k \widetilde{M_u}(\ga,k,\nu)$ is a graded algebra that will be called the \textit{graded algebra of all unrestricted unrestricted quasi-automorphic forms for the MS $\nu$}.

Let $f$ be a nonzero unrestricted automorphic form in $f\in M_u(\ga,k,\nu)$, $k\neq 0$. The derivative $df=f'$ satisfies
\begin{equation}\label{mod dervitive}
(f'|_k\g)(z)\,=\,\nu(\g)\left(kc(cz+d)^{k+1}f(z)\,+\,(cz+d)^{k+2}f'(z)\right) \,,\ z\in {\mathbb H}\,,\, \gamma=\binom{a\ b}{c\ d}\in \ga,
\end{equation}
that is, $df=f' \in \widetilde{M_u}(\ga,k+2,1,\nu)$ and
$$
P_{f',z}(X)=f'(z)+k f(z)X.
$$
If $L_f$ denote the logarithmic derivative of $f$, \ie
\begin{equation}\label{log dervitive}
L_f=f'/f,
\end{equation}
then
\begin{equation}\label{log quasi}
(L_f|_2\g)(z)=L_f(z)+ \frac{kc}{cz+d}\,,\ z\in {\mathbb H}\,,\, \gamma=\binom{a\ b}{c\ d}\in \ga.
\end{equation}
Thus $L_f \in \widetilde{M_u}(\ga,2,1,\nu=1)$, and
$$
P_{L_f,z}(X)=L_f(z)+k X.
$$
Furthermore, we have the following structure theorem.
\begin{thm}
With the above notations, we have
\begin{\enu}
\item $d\left(\widetilde{M_u}(\ga,k,p,\nu)\right)\subset\widetilde{M_u}(\ga,k+2,p,\nu)$.
\item If $\ga$ is a non elementary Fuchsian subgroup, then
$$\widetilde{M_u}(\ga,k,p,\nu)=\bigoplus_{n=0}^{n=p}M_u(\ga,k-2n,\nu)\, \displaystyle L_f^n.$$
\end{\enu}
\end{thm}
\begin{proof}
This is a straightforward generalization of the proof given  in \cite{RY}, where the author consider the case of holomorphic functions, the trivial MS, and the full modular group $\PZ$. Indeed, the only relevant property of $\PZ$ used there is of being non elementary.
\end{proof}

\begin{defn}

With the above notations, and If $f\in M(\ga,k,\nu)$, $k\neq 0$, then
\begin{enumerate}

\item  The space
$$
\widetilde{M}(\ga,k,p,\nu)=\bigoplus_{n=0}^{n=p}M(\ga,k-2n,\nu)\, \displaystyle L_f^n,
$$
is called the space of quasi-automorphic forms of weight $k$, depth $p$ and MS $\nu$ on $\ga$.\\

\item Suppose that $\ga$ is a non elementary Fuchsian subgroup. If $f$ has no zeros in $\BH$, then $L_f$ is holomorphic on $\BH$, and the spaces
$$
\widetilde{M_e}(\ga,k,p,\nu)=\bigoplus_{n=0}^{n=p}M_e(\ga,k-2n,\nu)\, L_f^n,
$$
$$
\widetilde{S}(\ga,k,p,\nu)=\bigoplus_{n=0}^{n=p}S(\ga,k-2n,\nu)\, L_f^n,
$$
are respectively called the space of entire quasi-automorphic forms, quasi-cusp forms of weight $k$, depth $p$ and MS $\nu$  on $\ga$.
\end{enumerate}
\end{defn}

 The group $\ga=\PZ$ is a non elementary Fuchsian group of the first kind. If $f=\Delta$ is the weight $12$ cusp form for $\PZ$, then $f$ does not vanish on $\BH$, and by \eqref{log-div} we have
$$L_f=2\pi i E_2.$$
Using the fact that the graded algebra of entire modular forms with the trivial MS on $\SZ$ is
$$\BC[E_4,E_6],$$
see \cite{RN},  we get
$$
\widetilde{M_{u}}(\PZ,\nu=1)=\BC[E_4,E_6,E_2]\,,
$$
where
$\widetilde{M_u}(\PZ,\nu=1)$, is the graded algebra of entire quasi-modular forms with the trivial MS on $\PZ$.
\begin{rem}
From now on, the trivial MS will be omitted from notations.
\end{rem}

\section{Construction of equivariant functions}
\label{construction of equi functions}
For a detailed discussion on equivariant functions for the modular group (construction and structure) see \cite{ElbSeb}.

Let $\ga$ be a subgroup of $\PR$, $k\in \BR-\{0\}$, and suppose that $f$ is a nonzero unrestricted automorphic form on $\ga$ of weight $k$ and MS $\nu$.
We attach to $f$ the meromorphic function
\begin{equation}\label{equiv}
h(z)\,=\,h_f(z)\,=\, z\,+\,k\frac{f(z)}{f'(z)}=  z\,+\,\frac{k}{L_f(z)}.
\end{equation}
\begin{prop}\cite{s-s,smart}\label{prop1} The function $h$ is an equivariant function, \ie,\, it satisfies
\begin{equation}\label{equiv1}
h(\gamma\cdot z)\,=\,\gamma\cdot h(z)\ \,\mbox{for all}\, \ \gamma\in \Gamma\ ,\ \, z\in{\mathbb H}.
\end{equation}
\end{prop}
\begin{proof}
Let $\displaystyle \gamma=\binom{a\ b}{c\ d}\in\ga$. Then, we have
$$\begin{array}{ccc}
h(\alpha\cdot z)&=&\displaystyle \frac{az+b}{cz+d}+\frac{\nu(\g)\,k(cz+d)^kf(z)}{\nu(\g)\left[ck(cz+d)^{k+1}f(z)+(cz+d)^{k+2}f'(z)\right]}\\
&&\\
&=&\displaystyle\frac{(az+b)(ckf(z)+(cz+d)f'(z))+kf(z)}{(cz+d)(ckf(z)+(cz+d)f'(z))}\\
&&\\
&=&\displaystyle\frac{akf(z)+(az+b)f'(z)}{ckf(z)+(cz+d)f'(z)}\
\end{array}
$$
using the identity $ad-bc=1$.
In the meantime, we have
$$\begin{array}{ccc}
\alpha\cdot h(z)&=&\displaystyle \frac{ah(z)+b}{ch(z)+d}\\
&&\\
&=&\displaystyle\frac{(az+b)f'(z)+akf(z)}{(cz+d)f'(z)+ckf}\ .
\end{array}
$$
\end{proof}

In fact, the key property in the above proof was the  form of the quasi-automorphic polynomial attached to $\frac{1}{k}L_f$, \ie,
$$
P_{\frac{1}{k}L_f,z}(X)=\frac{1}{k}L_f(z)+X,
$$
in the sense that  any meromorphic function $g$ on $\BH$ verifying the following  property
\begin{equation}\label{quasi trans}
(g|_2\g)(z)=g+ \frac{c}{cz+d}\,,\ z\in {\mathbb H}\,,\, \gamma=\binom{a\ b}{c\ d}\in \ga
\end{equation}
gives arise to an equivariant function given by
$$
h_g=h_0 +\frac{1}{g}\,,~~~~h_0(z)=z.
$$

\begin{defn}
Let $\widetilde M_u(\ga,2,1)$ be the space of unrestricted quasi-automorphic form of weight $2$, depth $1$ on $\ga$. If  $g_0$ is the constant function $g_0=\infty$, then we define
 $$\q=\{g_0\}\cup M_u(\ga,2,1),$$

\end{defn}
 Recall that the quasi-automorphic polynomial to any element $g\in \q\setminus\{g_0\}$ is given by
$$
P_{g,z}(X)=g(z)+X,
$$
Moreover, if $g \in\q$, and $f$ is any weight $2$ unrestricted automorphic form on $\ga$, that is
$$
f|_2\g=f,~\mbox{ for all } \g \in \ga,
$$
then one can easily check that
$$
f+g \in \q.
$$
Now, if we associate to each element $h\in\e$ the function defined by
$$
\widehat{h}=\displaystyle\frac{1}{h-h_0},
$$
then $\widehat{h} \in \q$, and $\widehat{h}_0=g_0$, and  clearly we have the following result.
\begin{prop}
We have a one-to-one correspondence between $\q $ and $\e$  given by the two mutually inverse maps
$$
g\ra h_g\,,~~~~~~~h\ra\widehat h.
$$
\end{prop}

\begin{proof}
The proof in \cite{ElbSeb} given for modular subgroups also works  for any subgroup of $\ga$ of $\PR$.
\end{proof}
\begin{eg} $ $
\item If $\ga=\PZ$ and $g= \frac{i \pi}{6}E_2$. Then any meromorphic equivariant function  $h$ of $\e$ has the form
$$
h=h_0+\frac{1}{f+ g},
$$
where $f$ is a unrestricted modular forms of weight $2$.
\end{eg}

We have seen in \eqref{log quasi} that the logarithmic derivative of an unrestricted automorphic form $f$ is an unrestricted quasi-automorphic form of weight 2 and depth 1. More explicitly, we have the following:\\
 If $f \in M_u(\ga,k,p,\nu),~k\neq0,$ then
$$h=h_0+\frac{k}{L_f} \in \e.$$

\begin{defn}
After \cite{ElbSeb}, the class of equivariant functions constructed above is referred to as the class of \textit{rational} equivariant functions.
\end{defn}

\section{Classification of unrestricted automorphic forms for non discrete subgroups of $\PR$ }
\label{9}
The classification of automorphic forms for a non discrete (non elementary) subgroup of $\PR$ was raised by many authors, see \cite{knopp1,Ber1,DJ} and
the references therein. Here we provide a new and simple proof of the complete classification of these forms.

\begin{lem}
If $\ga$ is a  non discrete subgroup $\ga$ of $\PR$, then the dimension of $ M_u(\ga,k,\nu)$ is less then one.
\end{lem}
\begin{proof}
Suppose that $f\in M_u(\ga,k,\nu)\setminus\{0\}$. If $g\in M_u(\ga,k,\nu)$ then $g/f$ is invariant under the action of $\ga$.
Since $\ga$ is not discrete, we have a sequence $\{\g_n\}$ of elements of $\ga$ such that $\g_n\ra I$ as $n\ra\infty$.
Take any point $x$ in $\BH \setminus Fix(\ga,\Cc)$ which is not a pole of $g/f$, and consider the meromorphic function $G = g/f -g(x)$ and the sequence $x_n= \g_n x$. Then we have
$$
G(x)=0,~G(x_n)=0, \mbox{ and } x_n\ra x \mbox{ as } n\ra\infty.
$$
Therefore, $G$ is meromorphic on $\BH$, and its set of zeros has an accumulation point in $\BH$. It follows that
$G$ must be identically zero, and so $g$ is a multiple of $f$. Hence the dimension of $ M_u(\ga,k,\nu)$ is one.

\end{proof}

\begin{thm}
If $\ga$ is a non elementary non discrete subgroup $\ga$ of $\PR$, then we have
\begin{\enu}
\item If $k\neq0$, then  $M_u(\ga,k,\nu)=\{0\}$.
\item If $k=0$, then  $M_u(\ga,k,\nu)= \BC$ when $\nu=1$, otherwise it is trivial.  
\end{\enu}
\end{thm}
\begin{proof}
Suppose that $k\neq0$, and that there is an element  $f\in M_u(\ga,k,\nu)-\{0\}$. Then
$$h_f=h_0+\frac{k}{L_f} \in \e\,.$$
Meanwhile, using \tf{th classification 1}, we have $\e=\{h_0\}$. Hence
$$\frac{k}{L_f}=0,$$
which is impossible. Thus, for a non elementary non discrete subgroup $\ga$ of $\PR$, $M_u(\ga,k,\nu)=\{0\}$.

If $k=0$ and  $f\in M'(\ga,0,\nu)$, then the derivative $f'$ of $f$ is an element of $M_u(\ga,2,\nu)$ and so $f'=0$, this implies that  $f$ is a constant  function. The conclusion is clear.
\end{proof}

\begin{rem}
Let $f$ be an unrestricted automorphic form on  $\ga$, and $L$ an element of  $\PR$. Then  $f|_kL$ is an unrestricted automorphic from  on $L^{-1}\ga L$. Hence it suffices to make classification up to a conjugation in $\PR$.
\end{rem}

Now suppose that $\ga$ is a non discrete elementary subgroup of $\PR$. If $\ga$ is an abelian, then by \tf{elementary groups}, it  contains beside the identity, only elliptic elements, or only parabolic elements, or only hyperbolic elements. After a suitable conjugation in $\PR$, we can suppose that
\begin{enumerate}
\item $Fix(\ga,\Cc)=\{i,\, -i\}.$
\item $Fix(\ga,\Cc)=\{0,\, \infty\}.$
\item $Fix(\ga,\Cc)=\{\infty\}.$

\end{enumerate}

\subsubsection*{{\bf  The case $\bf Fix(\ga,\Cc)=\{i,\, -i\}$:}}
$ $

By \tf{th classification 2}, we know that
$$\e=\{i,\, -i\}\, \cup \, \big\{ \g \in \PC;\, Fix(\g)=Fix(\ga,\Cc) \, \big\}.$$
If  $f$ is an element of $M_u(\ga,k,\nu)\setminus\{0\}$, $k\neq0$, then
$$h_0+\frac{k}{L_f}= h_f \in \e\,.$$
Suppose that $h_f \in \e\setminus\{i,\, -i\}$, and let $\lambda$ be the Cayley transformation, see Equation \ref{map1}. If  $w=\lambda(z)$, then
$$\begin{array}{ccc}
\displaystyle\frac{f'(z)}{f(z)}& = & \displaystyle k \frac{1}{h_f(z) -z} \\
&&\\
\displaystyle\frac{f'(\lambda^{-1}(w))}{f(\lambda^{-1}(w))}& = & \displaystyle k \frac{1}{h_f(\lambda^{-1}(w)) -\lambda^{-1}(w)} \\
&&\\
\displaystyle(w-1)^2\frac{(f(\lambda^{-1}(w)))'}{2if(\lambda^{-1}(w))}& = & \displaystyle k \frac{1}{\lambda^{-1}(\lambda h_f\lambda^{-1}(w)) -\lambda^{-1}(w)}.
\end{array}
$$
Hence, if $$g=f\lambda^{-1}\,,\mbox{ and}\,, h_g=\lambda h_f \lambda^{-1},$$
 then we have
$$
\frac{g'(w)}{g(w)}= k \frac{h_f -1}{(w-1)(h_f(w) -w)},~ \mbox{for all}~  w\in \BU.
$$
Since  $h_f$ fixes $i$ and $-i$, then there exists $\al$ in $\BC$ such that  $h_g(w)=\al w$, for all $w$ in $\BU$.
If $\al=1$, then $h_g$ is the identity map, and so $h_f=h_0$. But this implies that $f'/f=\infty$, which is absurd. If we suppose that $\al\neq 1$, then
$$
\frac{g'(w)}{g(w)}=k \frac{\al w -1}{(\al -1)w(w-1)}=  \frac{k}{(\al -1)w}+ \frac{k}{w-1} ,~ \mbox{for all}~  w\in \BU.
$$
Equating the residues at $0$ of the two expression, we get
$$
 \frac{k}{(\al -1)}=\mbox{order}_0(g)=\mbox{order}_0(f)=n_f,
$$
This implies that there exists $c$ in $\BC$ such that
$$
g(w)=c w^{n_f}(w-1)^k,~ \mbox{for all}~  w\in \BU,
$$
\ie,
$$
f(z)= c(\lambda(z))^{n_f}(\lambda(z)-1)^k,~ \mbox{for all}~  z\in \BH.
$$

In case $ h_f=-i$, then
$$
\frac{f'(z)}{f(z)} = -k \frac{1}{z+i},~ \mbox{for all}~  z\in \BH,
$$
and so
$$
f(z)=c (z+i)^{-k},~ \mbox{for all}~  z\in \BH,
$$
for some constant $c$ in $\BC$.

In case $ h_f=i$, then
$$
\frac{f'(z)}{f(z)} = -k \frac{1}{z-i},~ \mbox{for all}~  z\in \BH,
$$
Again by equating the residues at $i$ of the two expression, we find that $k$ is an integer, and that
$$
f(z)=c (z-i)^{-k},~ \mbox{for all}~  z\in \BH,
$$
for some constant $c$ in $\BC$.

\subsubsection*{{\bf  The case $ \bf Fix(\ga,\Cc)=\{0,\, \infty\}$:}}
$ $

By \tf{th classification 2}, we know that
$$\e=\{0,\, \infty\}\, \cup \, \big\{ \g \in \PC;\, Fix(\g)=Fix(\ga,\Cc) \, \big\}.$$
If  $f$ is an element of $M_u(\ga,k,\nu)\setminus\{0\}$, $k\neq0$, then
$$h_0+\frac{k}{L_f}= h_f \in \e\,.$$
Suppose that $h_f \in \e\setminus\{\infty\}$.  Then there exists $\al$ in $\BC$ such that  $h_f(z)=\al z$, for all $z$ in $\BH$.
As above,  the case  $\al=1$ is excluded.  If we suppose that $\al\neq 1$, then
$$
\frac{f'(z)}{f(z)} = k \frac{1}{h_f(z) -z},~ \mbox{for all}~  z\in \BH,
$$
$$
\frac{f'(z)}{f(z)} = k \frac{1}{(\al-1) z},~ \mbox{for all}~  z\in \BH,
$$
and so
$$
f(z)=c z^{\frac{k}{\al-1}},~ \mbox{for all}~  z\in \BH,
$$
for some constant $c$ in $\BC$.

In case $ h_f=\infty$, as above, we find that $f$ is a constant, and  the commutativity with $\ga$ implies that $f$ is the zero function, which is absurd.

\subsubsection*{{\bf  The case $\bf Fix(\ga,\Cc)=\{ \infty\}$:}}

By \tf{th classification 2}, we know that
$$\e=\{\infty\}\, \cup \, \big\{ \g \in \PC;\, Fix(\g)=Fix(\ga,\Cc) \, \big\}.$$
If  $f$ is an element of $M_u(\ga,k,\nu)\setminus\{0\}$, $k\neq0$, then
$$h_0+\frac{k}{L_f}= h_f \in \e\,.$$
Suppose that $h_f \in \e\setminus\{\infty\}$.  Then there exists $\al$ in $\BC$ such that  $h_f(z)= z + \al$, for all $z$ in $\BH$.
If $\al=0$, then $h_f=h_0$ which is absurd. If we suppose that $\al\neq 0$, then
$$
\frac{f'(z)}{f(z)} = k \frac{1}{h_f(z) -z},~ \mbox{for all}~  z\in \BH,
$$
$$
\frac{f'(z)}{f(z)} = \frac{k}{\al},~ \mbox{for all}~  z\in \BH,
$$
and so
$$
\displaystyle f(z)=c e^{(k/\al) z},~ \mbox{for all}~  z\in \BH,
$$
for some constant $c$ in $\BC$.

The case $ h_f=\infty$ is excluded since we will find that $f$ is a constant, and  the commutativity with $\ga$ will imply that $\ga=\{h_0\}$, thus contradicting the non discreteness of $\ga$.

We summarize the above discussion in the following theorem.
\begin{thm}
Let $\ga$ be an abelian  non discrete elementary subgroup of $\PR$, $k$ be a non zero reel number, $\lambda$ be the Cayley transformation.
After a suitable conjugation in $\PR$, we can suppose that
\begin{enumerate}
\item $Fix(\ga,\Cc)=\{i,\, -i\}.$
\item $Fix(\ga,\Cc)=\{0,\, \infty\}.$
\item $Fix(\ga,\Cc)=\{\infty\}.$
\end{enumerate}
If $M_u(\ga,k,\nu)$ is non trivial, then its dimension is $1$, and it is generated by an element $f$ having one of the following forms:
\begin{enumerate}
\item If $Fix(\ga,\Cc)=\{i,\, -i\}$, then for all $z$ in $\BH$
\begin{itemize} 
\item $\displaystyle f(z)=(\lambda(z))^{n_f}(\lambda(z)-1)^k$, $\displaystyle n_f=order_i(f)$.
\item $\displaystyle f(z)=(z+i)^{-k}$.
\item $\displaystyle f(z)=(z-i)^{-k}$.
\end{itemize}

\item If $Fix(\ga,\Cc)=\{0,\, \infty\}$,  then there exists an $\al$ in $\BC$, such that 
$$ f(z)=z^{\frac{k}{\al-1}},~ \mbox{for all}~  z\in \BH.$$

\item If $Fix(\ga,\Cc)=\{ \infty\}$,  then there exists an $\al$ in $\BC$, such that 
$$ \displaystyle f(z)=e^{(k/\al) z},~ \mbox{for all}~  z\in \BH.$$

\end{enumerate}
\end{thm}
Now, we come to the case when $\ga$ is non abelian non discrete elementary subgroup of $\PR$. Then by \tf{elementary groups}, either $\ga$ is conjugated in $\PR$ to a group generated by $h(z)=-1/z$ and by elements of the form $g(z)=lz$ ($l>1$). This case is characterized by 
$$\bigcap\limits_{\g \in \ga}  Fix(\g)=\{\emptyset\}.$$
 Or $\ga$ is conjugated in $\PR$ to a group whose elements have the form $az+b$, $a,b$ in $\BR$, and contains both parabolic and  hyperbolic elements. This case is characterized by $$\bigcap\limits_{\g \in \ga}Fix(\g)\neq\{\emptyset\}.$$

\begin{thm}
Let $\ga$ be an non abelian non discrete elementary subgroup of $\PR$, $k$ be a non zero reel number, $\nu$ a MS.
\begin{enumerate}
\item Suppose that $\bigcap\limits_{\g \in \ga}Fix(\g)=\{\emptyset\}$, and that $M_u(\ga,k,\nu)$ is non trivial. Then its dimension is $1$ and, after a conjugation in $\PR$, it is generated by:
$$\displaystyle f(z)=z^{\frac{-k}{2}},~ \mbox{for all}~  z\in \BH.$$
\item If $\bigcap\limits_{\g \in \ga}Fix(\g)\neq\{\emptyset\}$, then $M_u(\ga,k,\nu)$ is trivial.
\end{enumerate}
\end{thm}

\begin{proof}
The first case:  By  \tf{th classification 2}, we know  that after a conjugation  in $\PR$, we may assume that
$$\e= \big\{ h_0,\, -h_0 \big\}.$$

If  $f$ be an element of $M_u(\ga,k,\nu)\setminus\{0\}$, $k\neq0$, then
$$h_0+\frac{k}{L_f}= h_f \in \e\,.$$
As seen above, the case $h_f=h_0$ leads to a contradiction. If  $h_f=-h_0$ then
$$
\frac{f'(z)}{f(z)} = k \frac{1}{h_f(z) -z},~ \mbox{for all}~  z\in \BH,
$$
$$
\frac{f'(z)}{f(z)} =  \frac{-k}{2 z},~ \mbox{for all}~  z\in \BH,
$$
and so
$$
f(z)=c z^{\frac{-k}{2}},~ \mbox{for all}~  z\in \BH.
$$
for some constant $c$ in $\BC$.

The second case: By  \tf{th classification 2}, we have
$$\e= \big\{ h_0 \big\}.$$
If  $f$ be an element of $M_u(\ga,k,\nu)\setminus\{0\}$, $k\neq0$, then
$$h_0+\frac{k}{L_f}= h_f \in \e\,,$$
and so $h_f=h_0$, which leads to a contradiction. Hence the result.
\end{proof}

\section{Zeros of weight 2 depth 1  unrestricted quasi-automorphic forms }

We have the following result which is a direct consequence of \cf{cor1}.

\begin{thm}
Let $\ga$ be a Fuchsian group of the first kind, and suppose that $h \in \e-\{h_0\}$. Then for any $z\in \Cc$,
$h^{-1}(\{z\})$  contains infinitely many non  $\ga-$equivalent points.
In particular, for $z=\infty$, we find that  $h$ has infinitely many non $\ga$-equivalent poles in $\BH$.
\end{thm}

\begin{proof}
Let $h \in \e-\{h_0\}$. If $h$ has no poles in $\BH$, then it will be a holomorphic equivariant function for $\ga$. But \tf{th classification 3} says that  the unique holomorphic $\ga$-equivariant function is $h_0$ which contradicts our assumption. Thus $h$ has at least one pole $x \in \BH$.

Since $\ga$ is not elementary, it contains an element $\g=\binom{a\ b}{c\ d}$ with $c\neq 0$, otherwise all elements of $\ga$ will fix $\infty$, and so $\ga$ will be elementary. By the equivarience property we have
$$
h(\g x)=\g h(x)=\g\infty=\frac{a}{c}\in \Rr=\La.
$$
The theorem follows from \cf{cor1}
\end{proof}

\begin{thm}
Let $\ga$ be a Fuchsian group of the first kind, and $g\in \widetilde{M_u}(\ga,2,1)$ be a non zero unrestricted quasi-automorphic forms of weight $2$ and depth $1$ on $\ga$. Then $g$ has infinitely many non $\ga$-equivalent zeros in $\BH$.
\end{thm}

\begin{proof}

If $g\in \widetilde{M_u}(\ga,2,1)$ is a non zero unrestricted quasi-automorphic forms of weight $2$ depth $1$ on $\ga$, then
$$
h_g=h_0 +\frac{1}{g}
$$
is a meromorphic equivariant function. Since $g$ is meromorphic, we have $h_g\neq h_0$, and it is clear that the poles of $h_g$ are the zeros of $g$. The result follows from the above theorem.
\end{proof}

\section{Critical points of unrestricted automorphic forms }

The importance of the above theorems lies in the following striking result:
\begin{thm}
Let $\ga$ be a Fuchsian group of the first kind, and $f \in M_u(\ga,k,p,\nu),~k\neq0,$ be an unrestricted automorphic form of weight $k$ and MS $\nu$.
Then  $f$ has infinitely many non $\ga$-equivalent critical points in $\BH$.
\end{thm}

\begin{proof}
This follows from the fact that
$$h_f=h_0+k\frac{f}{f'} \in \e,$$
and that the zeros of $f$ are not poles of $h_f$ (they are fixed points of $h_f$). Thus the poles of $h_f$ are exactly the zeros of the derivative $f'$ of $f$.
\end{proof}

This theorem  can be thought  as a kind of an automorphy test:
\begin{cor}
Let $f$ be holomorphic function on $\BH$ such that
$$f(z)=\sum_{n\geq n_0}\,a_nq^n\ ,\, q=\exp(2\pi iz), \, z \in \BH.$$
If $\ga$ be a Fuchsian group of the first kind containing the translation $\gamma z= z+1,~  z\in {\mathbb H}$. Then a necessary condition for $f$ to be an automorphic form for $\ga$ is that the series
$$\sum_{n\geq n_0}\,na_nq^n\ ,\, q=\exp(2\pi iz), \, z \in \BH$$
has infinitely many non $\ga$-equivalent zeros in $\BH$.
\end{cor}
For example, a rational function in $q=\exp(2\pi iz)$ cannot be an automorphic form for a group larger then one generated by the above translation.

Also, a direct application of this automorphy test leads to the following result:

\begin{cor}
Let $\ga$ be a Fuchsian group of the first kind, and $f$ be a rational function of $z$ in $\BH$. Then $f$ is an an unrestricted automorphic form of some weight $k$ and MS $\nu$ if and only if $k=0$, $\nu=1$, and  $f$ is a constant.
\end{cor}

The above  automorphy test allowed us again to give  a new and simple proof of the following:

Suppose  $\ga$ is a Fuchsian group of the first kind, and that $\g$ is a parabolic element of $\ga$ with a fixed point $x$, and let
 $$L=L(x)=\left(
              \begin{array}{cc}
               0 &-1 \\
               1 & -x \\
              \end{array}
              \right).$$

Let $f \in M_u(\ga,k,p,\nu),~k\neq0$ be an unrestricted automorphic form of weight $k$ and MS $\nu$ which is  holomorphic in $\BH$. If $f$ is meromorphic at the cusp $x$, then
$$
(f|_kL)(z)= e^{2\pi ik_L z/n_L} \sum_{m=N_L}^{\infty} a_m(L)e^{2\pi i mz/n_L},~ N_L\in\BZ\cup\{\infty\}.
$$

When $f$ is a weight two cusp for a finite index subgroup $\ga$ of $\PZ$, $x=\infty$, and $N_L>0$, the authors of \cite{KohMas},  gave a negative answer to the following question:  is it possible that  $a_m(L)=0$ for $m\gg0\,? $ Their method uses some analytic properties of the \textit{Rankin-Selberg zeta function} of $f|_kL$, defined for $\Re(s)>2$ by
$$
R_{(f|_kL)}(s)=\sum_{m=1}^{\infty} \frac{a^2_m(L)}{m^{-s}}.
$$
They also gave another proof using the theory of \textit{vector-valued modular form}.
In our case, the answer is a simple consequence of the above theorem, without the restrictions on the group, the weight, and the MS required in their work.
\begin{thm}
Let $\ga$ be a Fuchsian group of the first kind, and $f \in M_u(\ga,k,p,\nu),~k\neq0,$ be an unrestricted automorphic form of weight $k$ and MS $\nu$. With the above notations, if $a_m(L)=0$ for $m\gg0$, then $f=0$.
\end{thm}
\begin{proof}
Suppose the converse is true, then $f|_kL$  will be a nonzero rational function of the local uniformizing parameter  $e^{2\pi iz/n_L}$, hence it cannot  have infinitely many non $L^{-1}\ga L$-equivalent zeros, which is absurd, since $L^{-1}\ga L$ is a Fuchsian group of the first kind, and $f|_kL \in M'(L^{-1}\ga L,k,p,\nu)$.
\end{proof}

\section{Examples}

In this last section, we give some examples involving  standard modular forms and their critical points.

%\begin{\enu}
%\item
$\bullet$ As a consequence of \eqref{log-div}, we recover some results of \cite{ElbSeb1,Wat}: The Eisenstein series $E_2$ (and hence $\Delta'$) has infinitely many non equivalent zeros in the strip $-1/2<\mbox{Re}(z)\leq /12$. Moreover, all these zeros are simple since $E_2$ and $E'_2$ cannot vanish at the same time because of \eqref{ram1} and the fact that $E_4$ vanishes only at the orbit of the cubic root of unity $\rho$. Similarly, the zeros of $E'_4$ are also simple because of \eqref{ram2} and the fact that $E_6$ vanishes only at the orbit of $i$. Using the same argument and \eqref{ram3}, the zeros of $E'_6$ are all simple.

We know that $E_2$ is real on the axis $\mbox{Re}(z)=0$ and $\mbox{Re}(z)=1/2$, and one can show that it has a unique zero on each axis given
approximately by $i \,0.5235217000$ and $1/2+i\, 0.1309190304$, see \cite{ElbSeb1}. Also, $E'_4$ and $E'_6$ are purely imaginary on both axis and have zeros on $\mbox{Re}(z)=1/2$ given respectively by $1/2+i\,0.4086818600$ and $1/2+i\,0.6341269863$.

$\bullet$ Using \eqref{log-div} and \eqref{eta},  we get the following formulas for the derivative of the theta functions:
$$
\frac{1}{4\pi i}\frac{\theta'_2(z)}{\theta_2(z)}\,=\,4E_2(4z)-E_2(2z),
$$

$$
\frac{24}{\pi i}\frac{\theta'_3(z)}{\theta_3(z)}\,=\,5E_2(2z)-E_2(z)-4E_2(4z),
$$

$$
\frac{1}{2\pi i}\frac{\theta'_4(z)}{\theta_4(z)}\,=\,E_2(z)-E_2(2z).
$$

It is interesting to notice that each of the above combinations with $E_2$, $E_4$ and $E_6$ vanishes
infinitely many times at inequivalent points in a vertical strip in $\mathbb H$.

$\bullet$ Let $F$ be any meromorphic function on $\BC$. then the function
$$
f(z)=j'(z)F'(j(z))+ \frac{i\pi}{6}E_2(z), ~z\in\BH,
$$
is an unrestricted quasi-modular form of weight 2 depth 1 for $\PZ$. Hence it has infinitely many non equivalent zeros in the strip $-1/2<\mbox{Re}(z)\leq 1/2$.
Note that the term
$$
j'(z)F'(j(z)), ~z\in\BH,
$$
corresponds to the change of variable $w=j(z)$, $w\in \BC$, and that even if the function $F(w)$ does not vanish on $\BC$ (e.g. $F(w)=e^{w}$), making this change of variable and performing a perturbation of $F$ by a fixed weight 2 depth 1 quasi-modular form leads to the intriguing fact: there are infinitely many zeros, and most importantly, they are non $\ga-$equivalent,  meaning that they don't come from the change of variable. A phenomenon that deserves to be investigated further.

\end{document}